\documentclass[reqno]{amsart}
\usepackage{amssymb}
\usepackage{amsthm}
\usepackage{bbm}
\usepackage{graphicx}

\usepackage[colorlinks=true]{hyperref}
\usepackage[all]{xypic}

\usepackage[final]{epsfig}
\usepackage{psfrag}
\usepackage{epstopdf}

\newtheorem{thm}{Theorem}[section]
\newtheorem{theorem}[thm]{Theorem}
\newtheorem{proposition}[thm]{Proposition}
\newtheorem{corollary}[thm]{Corollary}
\newtheorem{lemma}[thm]{Lemma}
\newtheorem{cor}[thm]{Corollary}
\newtheorem{lem}[thm]{Lemma}
\newtheorem{prop}[thm]{Proposition}

\theoremstyle{definition}
\newtheorem{definition}{Definition}[section]

\theoremstyle{observation}

\theoremstyle{definition}
\newtheorem{example}{Example}[section]

\newcommand{\B}{\mathcal{B}}

\newcommand{\defin}[1]{{\it #1}}

\newcommand{\N}{\mathbb{N}}
\newcommand{\Q}{\mathbb{Q}}

\theoremstyle{theorem}

{\bf}{\it}
\newtheorem*{bpt}{Beurling Projection Theorem}{\bf}{\it}
\newtheorem*{riemann-mapping-theorem}{Riemann Mapping  Theorem}
\newtheorem*{onequarter}{Koebe One-Quarter Theorem}{\bf}{\it}
\newtheorem*{distortion}{Koebe Distortion  Theorem}{\bf}{\it}
{\bf}{\it}
{\bf}{\it}
{\bf}{\it}
{\bf}{\it}
{\bf}{\it}
\newtheorem*{carat}{Carath{\'e}odory Theorem}{\bf}{\it}
\newtheorem*{baby-carat}{Carath{\'e}odory Theorem for locally connected domains}{\bf}{\it}
{\bf}{\it}
\newtheorem*{babythm}{Baby Theorem}{\bf}{\it}

\theoremstyle{remark}
\newtheorem{remark}[thm]{Remark}

\newenvironment{pf*}[1]{\proof[#1]}{\endproof}
\usepackage{euscript}

\newcommand{\cal}[1]{{\mathcal #1}}

\newcommand{\beq}{\begin{equation}}
\newcommand{\eeq}{\end{equation}}

\newtheorem{defn}{Definition}[section]

\newtheorem{rem}{Remark}[section]

\newcommand{\riem}{{\widehat{\CC}}}
\newcommand{\diam}{\operatorname{diam}}
\newcommand{\dist}{\operatorname{dist}}

\newcommand{\eps}{\epsilon}

\renewcommand{\Re}{\operatorname{Re}}
\renewcommand{\arg}{\operatorname{arg}}
\renewcommand{\Im}{\operatorname{Im}}

\numberwithin{equation}{section}
\newcommand{\thmref}[1]{Theorem~\ref{#1}}
\newcommand{\propref}[1]{Proposition~\ref{#1}}
\newcommand{\secref}[1]{\S\ref{#1}}
\newcommand{\lemref}[1]{Lemma~\ref{#1}}
\newcommand{\corref}[1]{Corollary~\ref{#1}}

\newcommand{\cI}{{\cal I}}
\newcommand{\cA}{{\cal A}}

\newcommand{\cB}{{\cal B}}

\newcommand{\cO}{{\cal O}}

\newcommand{\cS}{{\cal S}}

\newcommand{\CC}{{\mathbb C}}
\newcommand{\RR}{{\mathbb R}}

\newcommand{\ZZ}{{\mathbb Z}}
\newcommand{\NN}{{\mathbb N}}
\newcommand{\DD}{{\mathbb D}}

\newcommand{\QQ}{{\mathbb Q}}

\newcommand{\ov}[1]{\overline{#1}}

\newcommand{\ignore}[1]{{}}

\title{Computable Carath{\'e}odory Theory}
\author{Ilia Binder, Cristobal Rojas, Michael Yampolsky}
\thanks{I.B. and M.Y. were partially supported by NSERC Discovery Grants. C.R. was partially supported by project FONDECYT  No. 11110226. }
\date{\today}
\begin{document}
\begin{abstract}
Conformal Riemann mapping of the unit disk onto a simply-connected domain $W$  is a central object of study in classical Complex Analysis.
The first complete proof of the Riemann Mapping Theorem given by P.~Koebe in 1912 is constructive, and theoretical aspects of computing the
Riemann map have been extensively studied since. Carath{\'e}odory Theory describes the boundary extension of the Riemann map. In this paper
we develop its constructive version with explicit complexity bounds.

\end{abstract}

\maketitle
\section{Introduction}
Let $\DD=\{|z|<1\}\subset\CC$. The celebrated Riemann Mapping Theorem asserts:

\begin{riemann-mapping-theorem}
Suppose $ W\subset\CC$ is a simply-connected domain in the complex plane, and let $w$ be an arbitrary
point in $W$. Then there exists a unique conformal mapping
$$f:\DD\to W,\text{ such that }f(0)=w\text{ and }f'(0)>0.$$
\end{riemann-mapping-theorem}

\noindent
The inverse mapping,
$$\varphi\equiv f^{-1}: W\to\DD$$
is called the Riemann mapping of the domain $ W$ with base point $w$. The first complete proof of Riemann Mapping Theorem,
given by P. Koebe in 1912 \cite{Koebe} was constructive. Computation of the mapping $\varphi$ is important for applications, and
numerous algorithms have been implemented in practice (see \cite{Mar,MR}). Theoretical aspects of computing the Riemann mapping
were studied exhaustively in \cite{BBY3}, where a precise complexity bound on such algorithms was established.

The theory of Carath{\'e}odory (see e.g. \cite{Mil,Pom75})
 deals with the question of extending the map $f$ to the unit circle. It is most widely known
in the case when $\partial W$ is a locally connected set. We remind the reader that a Hausdorff topological space $X$ is
called locally connected if for every point $x\in X$ and every open set $V\ni x$ there exists a connected set $U\subset V$
such that $x$ lies in the interior of $U$. Thus, every point $x\in X$ has a basis of connected, but not necessarily open,
neighborhoods. This condition is easily shown to be equivalent to the (seemingly stronger) requirement that every point $x\in X$
has a basis of open connected neighborhoods. In its simplest form, Carath{\'e}odory Theorem says:

\begin{baby-carat}
A conformal mapping $f:\DD\to W$ continuously extends to the unit circle if and only if $\partial W$ is locally connected.
\end{baby-carat}

\noindent
A natural question from the point of view of Computability Theory is then the following:

\medskip
\noindent
{\sl Suppose the boundary of the domain $ W$ is described in some constructive fashion. Can the Carath{\'e}odory extension
$f:S^1\to\partial W$ be computed?}

\medskip
\noindent
In this paper we will do a lot more than give an answer to the above question -- we will build a constructive Carath{\'e}odory
theory for a general domain $ W$ with explicit complexity bounds. Before we can proceed with it, we need to give a brief
introduction to Computability Theory over the reals (\secref{sec:intro-comp}). We introduce Carath{\'e}odory Theory in \secref{sec:carat}, and
provide some necessary background from Complex Analysis in \secref{sec:complex}.

We note that there have been several previous attempts to formulate a computable Carath{\'e}odory Theorem for locally connected
domains \cite{mcn}. We have found them lacking in both the generality of the statements and in mathematical rigour of the proofs;
the approach we take is completely independent.

 \section{Introduction to Computability}
\label{sec:intro-comp}
\subsection{Algorithms and computable functions on integers}
The notion of an algorithm was formalized in the 30's,
independently by Post, Markov, Church, and, most famously, Turing. Each of them
proposed a model of computation which determines a set of integer functions that can be \emph{computed} by some mechanical or algorithmic procedure. Later on, all these models were shown to be equivalent, so that they define the same class of integer
functions, which are now called \emph{computable (or recursive) functions}.
It is standard in Computer Science to formalize
an algorithm as a {\it Turing Machine} \cite{Tur}. We will not define it here, and instead will refer an interested
reader to any standard introductory textbook in the subject. It is more intuitively familiar, and provably equivalent,
to think of an algorithm as a program written in any standard programming language.

In any programming language there is only a countable number of possible algorithms. Fixing the language, we can
enumerate them all (for instance, lexicographically). Given such an ordered list $(\cA_n)_{n=1}^\infty$ of all
algorithms, the index $n$ is usually called the \emph{G\"odel number} of the algorithm $\cA_{n}$.

We will call a  function $f:\NN\to\NN$  \emph{computable} (or {\it recursive}), if there exists an algorithm $\cA$ which, upon input $n$, outputs $f(n)$. Computable functions of several integer variables are defined in the same way.

A function $f:W\to\NN$, which is defined on a subset $W\subset \NN$, is called {\it partial recursive} if  there exists an
algorithm $\cA$ which outputs $f(n)$ on input $n\in W$, and runs forever if the input $n\notin W$.

\subsection{Time complexity of a problem.} For an algorithm $\cA$ with input $w$ the {\it running time}  is the number
of steps $\cA$ makes before terminating with an output.
The {\it size} of an input $w$ is the number of dyadic bits required to specify $w$. Thus for $w\in\NN$, the size of $w$ is
the integer part of $\log_2 l(w)$, where $l(w)$ denotes the length of $w$.
The {\it running time of $\cA$} is
the function
$$T_\cA :\NN \to\NN$$
 such that
$$T_\cA(n) = \max\{\text{the running time of }\cA(w)\text{ for inputs } w \text{ of size }n\}.$$
In other words, $T_\cA(n)$ is the worst case running
time for inputs of size $n$.
For a computable function $f : \NN\to \NN$  the time complexity of $f$ is said to have an
upper bound $T(n)$ if there exists an algorithm $\cA$ with running time bounded
by $T(n)$ that computes $f$. We say that the time complexity of $f$ has a lower bound $T(n)$ if for every algorithm $\cA$ which computes $f$,
there is a subsequence $n_k$ such that the running time $$T_\cA(n_k)>T(n_k).$$

\subsection{Computable and semi-computable sets of natural numbers}
A set $E\subseteq \NN$ is said to be \defin{computable} if its characteristic function
$\chi_E:\NN\to\{0,1\}$ is computable. That is, if there is an algorithm $\cA:\NN \to \{0,1\}$ that, upon input $n$,  halts and outputs $1$ if $n\in E$ or $0$ if $n\notin E$. Such an algorithm allows to \emph{decide} whether or not a number $n$ is an element of $E$. Computable sets are also called \defin{recursive} or \defin{decidable}.

Since there are only countably many algorithms, there exist only countably many computable subsets of $\NN$. A well known ``explicit'' example of a non computable set is given by the \emph{Halting set}
$$H:=\{i \text{ such that }\cA_{i}\text{ halts}\}.$$
 Turing \cite{Tur} has shown that  there is no algorithmic procedure to decide, for any $i\in\NN$, whether or not the algorithm with G\"odel number $i$, $\cA_{i}$, will eventually halt.

  On the other hand, it is easy to describe an algorithmic procedure which, on input $i$, will halt if $i\in H$, and will run forever if $i\notin H$. Such a procedure can informally be described as follows: {\sl on input $i$ emulate the algorithm $\cA_i$; if $\cA_i$ halts then  halt.}

  In general, we will say that a set $E\subset \NN$ is
  \defin{lower-computable} (or
  \defin{semi-decidable}) if there exists an algorithm $\cA_E$
  which on an input $n$  halts if $n\in E$, and never halts otherwise. Thus, the algorithm $\cA_E$
  can verify the inclusion $n\in E$, but not the inclusion $n\in E^c$.
We say that $\cA_E$ {\it semi-decides} $n\in E$ (or semi-decides $E$).
 The complement of a lower-computable set is called \defin{upper-computable}.

It is elementary to verify that lower-computability is equivalent to \defin{recursive enumerability}:
\begin{prop}
A set $E\subset\NN$ is lower-computable if and only if there exists an algorithm $\cA$ which outputs a sequence
of natural numbers $(n_i)$ such that $E=\cup\{ n_i\}$.
\end{prop}
We say that $\cA$ {\it enumerates} $E$.

 The following is an easy excercise:
 \begin{prop}
 A set is computable if and only if it is simultaneously upper- and lower-computable.
 \end{prop}

Note that the Halting set $H$ is an example of a lower-computable non-computable set.

\subsection{Computability over the reals}  Strictly speaking, algorithms only work on natural numbers, but this can be easily extended to the objects  of any countable set once a bijection with integers has been established.  The operative power of an algorithm on the objects of such a numbered set obviously depends on what can be algorithmically recovered from their numbers.  For example,  the set $\mathbb{Q}$ of rational numbers can be injectively
numbered $\mathbb{Q}=\{q_0,q_1,\ldots\}$ in an \emph{effective} way: the
number $i$ of a rational $a/b$ can be computed from $a$ and $b$, and vice
versa. The abilities of algorithms on integers are then transferred to the rationals. For instance, algorithms can perform algebraic operations and decide whether or not $q_{i}>q_{j}$ (in the sense that the set $\{(i,j):q_{i}>q_{j}\}$ is decidable).

Extending algorithmic notions to functions of real numbers was pioneered by Banach and Mazur \cite{BM,Maz}, and
is now known under the name of {\it Computable Analysis}. Let us begin by giving the modern definition of the notion of computable real
number,  which goes back to the seminal paper of Turing \cite{Tur}.

\begin{definition}A real number $x$ is called
 \defin{computable} if there is a computable function $f:\NN \to \QQ$ such that $$|f(n)-x|<2^{-n};$$
A point in $\RR^n$ is computable if all its coordinates are computable real numbers. A point $z\in\CC$ is computable
if both $\Re z$ and $\Im z$ are computable.
\end{definition}

Algebraic numbers or  the familiar constants such as $\pi$, $e$, or the Feigembaum constant are all computable. However, the set of all computable numbers $\RR_C$ is necessarily
countable, as there are only countably many computable functions.

\subsection{Uniform computability} In this paper we will use algorithms to define \emph{computability} notions on more general objects. Depending on the context, these objects will take particulars names (computable, lower-computable, etc...) but the definition will always follow the scheme:

\medskip
\noindent
\emph{an object $x$ is \emph{computable} if there exists an algorithm $\cA$
satisfying the \emph{property} P($\cA,x$)}.

\medskip\noindent
For example, a real number $x$ is \emph{computable} if there
exists an algorithm $\cA$ which computes a function $f:\NN\to \QQ$ satisfying $|f(n)-x|<2^{-n}$ for all $n$.
 Each time such definition is made, a \emph{uniform version} will be implicitly defined:

\medskip
\noindent
\emph{the objects $\{x_\gamma\}_{\gamma\in\Gamma}$ are computable \defin{uniformly on
a countable set $\Gamma$} if there exists an algorithm $\cA$ with an input $\gamma\in\Gamma$, such that for all
$\gamma\in \Gamma$, $\cA_\gamma:=\cA(\gamma,\cdot)$ satisfies the \emph{property} P($\cA_\gamma,x_\gamma$)}.

\medskip
\noindent
 In our example, a sequence of reals $(x_i)_i$ is \emph{computable uniformly in $i$} if there exists
$\cA$ with two natural inputs $i$ and $n$ which computes a function $f(i,n):\NN\times\NN\to\QQ$ such that for all $i\in \NN$, the values of the function $f_{i}(\cdot):=f(i,\cdot)$ satisfy

$$|f_{i}(n)-x_{i}|<2^{-n}\text{ for all } n \in\NN.$$

\subsection{Computable metric spaces}The above definitions equip the real numbers with a computability structure. This can be extended to virtually any separable metric space, making them \emph{computable metric spaces}. We now give a short introduction. For more details, see \cite{Wei}.

\begin{definition}
A \defin{computable metric space} is a triple $(X,d,\cS)$ where:
\begin{enumerate}
\item $(X,d)$ is a separable metric space,
\item $\cS=\{s_i:i\in\N\}$ is a dense sequence of points in $X$,
\item $d(s_i,s_j)$ are computable real numbers, uniformly in $(i,j)$.
\end{enumerate}
\end{definition}

The  points in $\cS$ are called  \defin{ideal}.
\begin{example}
A basic example is to take the space $X=\RR^n$ with the usual notion of Euclidean distance $d(\cdot,\cdot)$, and
to let the set $\cS$ consist of points $\overline x=(x_1,\ldots,x_n)$ with rational coordinates. In what follows,
we will implicitly make these choices of $\cS$ and $d(\cdot,\cdot)$ when discussing computability in $\RR^n$.
\end{example}

\begin{definition}A point $x$ is \defin{computable} if there is a computable function
$f:\NN\to\NN$  such that $$|s_{f(n)}-x|<2^{-n}\text{ for all }n.$$
\end{definition}

If $x\in X$ and $r>0$, the metric ball $B(x,r)$ is defined as
$$B(x,r)=\{y\in X:d(x,y)<r\}.$$
Since the set $\cB:=\{B(s,q):s\in\cS,q\in\Q, q>0\}$ of \defin{ideal balls} is countable, we can fix an enumeration
$\B=\{B_i:i\in\N\}$.

\begin{proposition}\label{p.comp.semidec}
A point $x$ is computable if and only if the relation $x\in B_{i}$ is semi-decidable, uniformly in $i$.
\end{proposition}
\begin{proof}
Assume first that $x$ is computable. We have to show that there is an algorithm $\cA$ which inputs a natural number $i$ and halts
 if and only if $x\in B_{i}$. Since $x$ is computable, for any $n$ we can produce an ideal point $s_{n}$ satisfying $|s_{n}-x|<2^{-n}$.
The algorithm $\cA$ work as follows: upon input $i$, it computes the center and radius of $B_{i}$, say $s$ and $r$. It then
searches for $n\in\NN$ such that
$$d(s_{n},s)+2^{-n}<r.$$
Evidently, the above inequality will hold for some $n$ if and only if $x\in B_i$.

Conversely, assume that the relation  $x\in B_{i}$, s semi-decidable uniformly in $i$.
To produce an ideal point $s_{n}$ satisfying $|s_{n}-x|<2^{-n}$, we only need to enumerate all ideal balls of radius $2^{-{n+1}}$ until one containing
$x$ is found. We can take $s_{n}$ to be the center of this ball.
\end{proof}

\begin{definition}
 An open set $U$ is called \defin{lower-computable} if there is a computable function $f:\NN\to\NN$
 such that $$U=\bigcup_{n\in \NN}B_{f(n)}.$$
\end{definition}

\begin{example} Let $\eps>0$ be a lower-computable real. Then the ball $B(0,\eps)$ is a lower-computable open set.
Indeed: $B(s_0,\eps)=\bigcup_{n}B(0,q_{n})$, where $(q_{n})_{n}$ is the computable sequence converging to $\eps$ from below.
\end{example}

It is not difficult to see that finite intersections or infinite unions of (uniformly) lower-computable open sets are again lower computable.
As in Proposition (\ref{p.comp.semidec}), one can show that the relation $x\in U$ is semi-decidable for a computable point $x$ and
an open lower-computable set.

We will now introduce computable functions. Let $X'$ be another computable metric space with ideal balls $\cB'=\{B_i'\}.$

\begin{definition}A function $f:X\to X'$ is \defin{computable} if the sets $f^{-1}(B'_i)$ are lower-computable open, uniformly in $i$.
\end{definition}

An immediate corollary of the definition is:

\begin{prop}
Every computable function is continuous.
\end{prop}

The above definition of a computable function is concise, yet not very transparent. To give its $\eps-\delta$ version, we need another concept.
We say that a function $\phi:\NN\to\NN$ is an {\it oracle} for $x\in X$ if
$$d(s_{\phi(m)},x)<2^{-m}.$$
An algorithm may {\it query} an oracle by reading the values of the function $\phi$ for an arbitrary $n\in\NN$.
We have the following:

\begin{prop}
\label{defcomp}
A function $f:X\to X'$ is computable if and only if
there exists an algorithm $\cA$ with an oracle for $x\in X$ and an input $n\in\NN$ which outputs $s_n'\in \cS'$
such that $d(s_n',f(x))<2^{-n}.$
\end{prop}
In other words, given an arbitrarily good approximation of the input of $f$ it is possible to constructively
approximate the value of $f$  with any desired precision.

For a computable function $f : X\to X$  the time complexity of $f$ is said to have an
upper bound $T(n)$ if there exists an algorithm $\cA$ with an oracle for $x\in X$ as described in \propref{defcomp} with running time bounded
by $T(n)$ for any $x\in X$ and any oracle $\phi$ for $x$.
 We say that the time complexity of $f$ has a lower bound $T(n)$ if for every such algorithm $\cA$, there exists $x\in X$ and an oracle $\phi$ for $x$ such that the running time of $\cA^\phi$ on input $n_k$  is at least $T(n_k).$

For a  function $f:X\to Y$ between metric spaces, we say that $h(\eps)$ is a {\it modulus of fluctuation} if 
$$\dist_Y(f(y),f(x))<h(\eps)\text{ whenever }\dist_X(y,x)<\eps.$$
Note that a continuous function from a compact metric space possesses a minimum  modulus of fluctuation: 
$$
h(\eps):=\sup_{\dist_{X}(x,x')\leq\eps}\dist_{Y}(f(x),f(x'))
$$
which satisfies $h(\eps)\searrow 0$ as $\eps\searrow 0$.

In this case, we will say that a function  $\mu:\NN \to \NN$ is the \textit{modulus of continuity} of $f$ if it is the smallest non decreasing function satisfying
$$
h(2^{-\mu(k)})\leq 2^{-k} \quad \text{ for all }\,\, k\in \NN. 
$$

%
%

We will make use of the following connection between time complexity and modulus of continuity of a function $f$. 

\begin{prop}
\label{modcon2}
Let $f:K\subset \RR^n\to\RR^m$ be a computable function from a compact set and let 
$\mu(k)$ be its modulus of continuity. Then the computational complexity of $f$ is bounded from below by $\mu(k)$. 
\end{prop}

\begin{proof}Let $x\in K$. 
Let $\phi(k)$ be an oracle for $x$ which on input $k$ outputs the $k$-bit binary approximation of $x$. Suppose the running time is given by $T(k)$. It takes 
 $T(k)$ computation steps to read $T(k)$ bits. Therefore, when computing the value of $f(x)$ with precision $2^{-k}$,
 the algorithm will know $x$ with precision at most $2^{-T(k)}$. Thus on a disk of diameter $2^{-T(n)}$ the value of $f(x)$ 
 cannot fluctuate by more than $2^{-k}$. But $\mu(k)$ is the smallest non decreasing function with this property. 
\end{proof}

\subsubsection{Computability of closed sets}
Having defined lower-computable open sets, we naturally proceed to the following definition.
\begin{definition}\label{d.upper.comp}
A closed set $K$  is  \defin{upper-computable}  if its complement is lower-computable. \end{definition}

\begin{definition}
A closed set $K$ is \defin{lower-computable} if the relation $K\cap B_{i}\neq \emptyset$ is semi-decidable, uniformly in $i$.
\end{definition}

\noindent
In other words, a closed set $K$ is lower-computable if there exists an algorithm $\cA$ which enumerates all ideal balls which have  non-empty intersection with $K$.

To see that this definition is a natural extension of lower computability of open sets, we note:
\begin{example}$\-$
\begin{enumerate}

\item  The closure of an ideal ball $\overline{B(s,q)}$ is lower-computable. Indeed, $B(s_{i},q_{i})\cap \overline{B(s,q)}\neq \emptyset $ if and only if $d(s,s_{n})<q+q_{n}$.

\item More generally, the closure $\ov{U}$ of any open lower-computable set $U$ is lower-computable since $B_{i}\cap \ov{U}\neq \emptyset$ if and only if there exists $s\in  B_{i}\cap U$.

\end{enumerate}
\end{example}

The following is  a useful characterization of lower-computable sets:
\begin{proposition}\label{p.comp-closed}
A closed set $K$ is lower-computable if and only if  there exists a sequence of uniformly computable  points $x_{i}\in K$ which is dense in $K$.
 \end{proposition}
\begin{proof}
Observe that, given some ideal ball $B=B(s,q)$ intersecting $K$, the relations $\ov{B_i}\subset B$, $ q_{i}\leq 2^{-k}$ and $B_{i}\cap K\neq \emptyset$ are all semi-decidable and then we can find an exponentially decreasing sequence of ideal balls $(B_{k})$ intersecting $K$. Hence $\{x\}=\cap_k B_k$ is a computable point lying in $B\cap K$.

The other direction is obvious.
\end{proof}

\begin{definition}A closed set is \defin{computable} if it is lower and upper computable.
\end{definition}

Here is an alternative way to define a computable set. Recall that {\it Hausdorff distance} between two
compact sets $K_1$, $K_2$ is
$$\dist_H(K_1,K_2)=\inf_\eps\{K_1\subset U_\eps(K_2)\text{ and }K_2\subset U_\eps(K_1)\},$$
where $U_\eps(K)=\bigcup_{z\in K}B(z,\eps)$ stands for an $\eps$-neighborhood of a set.
The set of all compact subsets of $M$ equipped with Hausdorff distance is a metric space which we will denote
by $\text{Comp}(M)$. If $M$ is a computable metric space, then $\text{Comp}(M)$ inherits this property; the ideal
points in $\text{Comp}(M)$ are finite unions of closed ideal balls in $M$. We then have the following:

\begin{prop}
\label{comp-set-1}
A set $K\Subset M$ is computable if and only if it is a computable point in $\text{Comp}(M)$.
\end{prop}

\begin{prop}
\label{comp-set-2}
Equivalenly, $K$ is computable if there exists an algorithm $\cA$ with a single natural input $n$, which
outputs a finite collection of closed ideal balls $\ov{B_1},\ldots, \ov{B_{i_n}}$ such that
$$\dist_H(\bigcup_{j=1}^{i_n} \ov{B_j}, K)<2^{-n}.$$
\end{prop}

We recall (see for example, Theorem 5.1 from\cite{Her}):
\begin{thm}
\label{comp-riem-map}
Let $W\subset \CC$ be a simply-connected domain. Then, the following are equivalent:
\begin{itemize}
\item[(i)] $W$ is a lower-computable open set, $\partial W$ is a lower-computable closed set,
and $w\in W$ be a computable point;
\item[(ii)] The map $$f:\DD\to W, \, f(0)=w,\;f'(0)>0;\,\, \text{and its inverse}\,\, \varphi\equiv f^{-1}: W\to\DD, $$
are both computable conformal bijections.
\end{itemize}
\end{thm}

\subsubsection{Computably compact sets}

\begin{definition}
A set $K\subseteq X$ is called \defin{computably compact} if it is compact and there exists an algorithm $\cA$ which on input $(i_{1},\ldots ,i_{p})$ halts if and only if
 $(B_{i_{1}},\ldots ,B_{i_{p}})$ is a covering of $K$.
\end{definition}

In other words, a compact set $K$ is computably compact if we can semi-decide  the inclusion $K\subset U$, uniformly from a description of the open set $U$ as a lower computable open set.

As an example, we note that using  Proposition \ref{p.comp.semidec}, it is easy to see that a  singleton $\{x\}$ is a computably compact set if and only if  $x$ is a computable point.

\begin{proposition}\label{p.subcompact}Assume the space $X$ is computably compact. Then a closed subset $E$ is computably compact if and only if $E$ is upper-computable.
\end{proposition}
\begin{proof}
We show that for a lower-computable open set $U$, the relation $E\subset U$ is semi-decidable uniformly from a description of $U$. Since $E$ is upper computable, its complement $X\setminus E$ is a lower-computable open set. Therefore, the set $U \cup X\setminus E$ is lower computable as well.  The relation $E\subset U$ is equivalent to $X\subset U \cup X\setminus E$, which is semi-decidable since $X$ is computably compact.
\end{proof}

The following result will be used in the sequel.

\begin{theorem}
\label{invert}
Let $f: X \to Y$ be a computable bijection between computable metric spaces. If $X$ is computably compact, then $f^{-1}:Y \to X$ is computable.
\end{theorem}
\begin{proof}
Let $y\in Y$. We exhibit an algorithm $\cA$ which in presence of an oracle for $y$, computes arbitrarily good approximations of $x\in X$ such that $f(x)=y$.  Start by  computing a description of $Y\setminus\{y\}=\bigcup_{y\in B_{i}} Y\setminus \overline{B_{i}}$ as a lower computable open set.  This is possible because from a description of $y$, one can semi-decide $y\in B_{i}$, uniformly in $i$.  Since $f$ is computable, one can compute a description of the  set  $f^{-1}(Y\setminus \{y\})$ as a lower-computable open set. This gives us, in particular, a description of its complement $\{x\}$ as an upper computable closed set.  The result now follows from  Proposition \ref{p.subcompact} and the fact that a singleton $\{x\}$ is computably compact iff it is a computable point.
\end{proof}

We recall computability of suprema over  computably compact spaces.

\begin{proposition}\label{modcomp}
Let $f:X\to \RR$ be a computable function. If $X$ is computably compact, then  $\bar{x}:=\sup_{x\in X}f(x)$ is computable.
\begin{proof}
The sequence $f(s_{i})$ where the $s_{i}$'s are the ideal points of $X$ is computable and $\bar{x}=\limsup f(s_{i})$, so that $\bar{x}$ is lower-computable. To see that $\bar{x}$ is also upper-computable, note that for $q\in \QQ$, the relation $\bar{x}<q$ is equivalent to $X\subset f^{-1}(-\infty,q)$, and therefore semi-decidable.
\end{proof}
\end{proposition}

\subsection{Conditional computability results}
For a real number $\alpha$ the computability of the closed Euclidean ball
$\ov{B(0,\alpha)}\subset \RR^m$ is clearly equivalent to the computability
of $\alpha$ itself. However, we may want to separate the problem of computing the radius of a Euclidean ball (the number $\alpha$) from
the problem of computing the ball when the radius is given. Intuitively, the latter should always be possible. To formalize this,
we can use the concept of oracle, as described above:

\begin{babythm}
There exists an algorithm $\cA$ with an oracle $\phi$ for a real number $\alpha$ which takes a natural number $n$ as an input,
and outputs a closed ball $\ov{B_n}$ such that
$$\dist_H(\ov{B_n}, \ov{ B(0,\alpha)})<2^{-n}.$$
\end{babythm}
\begin{proof}
The algorithm $\cA$ works as follows:\\
{\bf let} $r_n=\phi(n+1)$ (so that $|r_n-\alpha|<2^{-(n+1)}$;\\
{\bf output} $\ov{B_n}=\ov{ B(0,r_n)}.$
\end{proof}

\noindent
We will say that the ball of radius $\alpha$ is {\it conditionally computable} with an oracle for $\alpha$. Most of the computability
results in this paper will be stated as conditional theorems.

\section{Carath{\'e}odory Theory}
\label{sec:carat}
\subsection{Carath{\'e}odory Extension Theorem}
We give a very brief account of the principal elements of the theory here, for details see e.g. \cite{Mil,Pom75}.
In what follows, we fix a simply connected domain $ W\subset\CC$, and a point $w\in W$; we will
refer to such a pair as a {\it pointed domain}, and use notation $( W,w)$.
A {\it crosscut} $\gamma\subset W$ is a homeomorphic image
of the open interval $(0,1)$ such that the closure $\overline\gamma$ is homeomorphic to the closed inerval $[0,1]$ and the two endpoints
of $\overline\gamma$ lie in $\partial W$. It is not difficult to see that a crosscut divides
$ W$ into two connected components. Let $\gamma$ be a crosscut such that $w\notin\gamma$.
The component of $ W\setminus\gamma$ which does not contain $w$ is called {\it the crosscut neighborhood} of $\gamma$ in
$( W,w)$. We will denote it $N_\gamma$.

A {\it fundamental chain} in $( W,w)$ is a nested infinite sequence
$$N_{\gamma_1}\supset N_{\gamma_2}\supset N_{\gamma_3}\supset\cdots$$
of crosscut neighborhoods such that the closures of the crosscuts $\gamma_j$ are disjoint, and such that
$$\diam \gamma_j\longrightarrow 0.$$
Two fundamental chains $(N_{\gamma_j})_{j=1}^\infty$ and $(N_{\tau_j})_{j=1}^\infty$ are {\it equivalent} if every $N_{\gamma_j}$ contains
some $N_{\tau_i}$ and conversely, every $N_{\tau_i}$ contains
some $N_{\gamma_j}$.
Note that any two fundamental chains $(N_{\gamma_j})_{j=1}^\infty$ and $(N_{\tau_j})_{j=1}^\infty$ are either equivalent or eventually
disjoint, i.e. $N_{\gamma_j}\cap N_{\tau_i}=\emptyset$ for $i$ and $j$ sufficiently large.

The key concept of Carath{\'e}odory theory is a {\it prime end}, which is an equivalence class of fundamental
chains. The {\it impression} $\cI(p)$ of a prime end $p$ is a compact connected subset of $\partial  W$ defined as
follows: let  $(N_{\gamma_j})_{j=1}^\infty$ be any fundamental chain in the equivalence class $p$, then
$$\cI(p)=\cap\overline{N_{\gamma_j}}.$$
We say that the impression of a prime end is {\it trivial} if it consists of a single point. It is easy to see (cf. \cite{Mil}) that:
\begin{prop}
\label{trivial-impression}
If the boundary $\partial W$ is locally connected then the impression of every prime end is trivial.

\end{prop}
\begin{proof}
Assume first that $\partial W$ is locally connected.
By compactness of $\partial W$, for every $\eps>0$, we can select $\delta>0$ so that any two points of distance $<\delta$ are
contained in a connected subset of $\partial W$ of diameter $<\eps$.

Let $(N_{\gamma_j})_{j=1}^\infty$ be a fundamental chain. For $\eps$ and $\delta$ as above, choose $j$ large enough so that
the crosscut $\gamma_j$ has diameter less than $\delta$. It follows that there is a compact connected subset $Y\subset \partial W$
which contains the endpoints of $\gamma_j$. It is easy to see that the set $Y\cup\gamma_j$ separates neighborhood $N_{\gamma_j}$ from
$W\setminus \ov{N_{\gamma_j}}.$ Indeed, otherwise we could select a simple smooth curve $c_1$ which is disjoint from  $Y\cup\gamma_j$
and joins some point $x\in N_{\gamma_j}$ to $y\in W\setminus \ov{N_{\gamma_j}}.$ Adjoining a suitably chosen smooth arc $c_2\subset W$
 from $x$ to $y$ which
cuts once across $\gamma_j$, we obtain a Jordan curve $c_1\cup c_2$ which separates the two endpoints of $\gamma_j$. Hence, it separates
the set $Y$ which is impossible, since it was assumed to be connected.

The diameter of $Y\cup\gamma_j$ is less than $\eps+\delta$. If $\eps+\delta$ is small enough, it follows that $N_j$ must have diameter
less than $\eps+\delta$. Since $\eps$ and $\delta$ can be arbitrarily small, it follows that $\cap \ov N_j$ is a single point.

\end{proof}

\noindent
As an example, consider the simply-connected domain $ W$ around the origin, whose boundary is obtained by adjoining to
the unit circle $S^1$ the radial segments $S_\theta=\{re^{2\pi i\theta}|\;0.5\leq r\leq 1\}$ for $\theta=\frac{1}{n},$ $n\in\ZZ$ and $\theta=0$.
It is easy to see that $\partial W$ is not locally connected at every point of the segment
$S_0$, which is the accumulation of the ``double comb'' (see Figure \ref{fig-example}). There is a single prime end whose impression is all of $S_0$,
and the impressions of all other prime ends are trivial.

\begin{figure}
\centerline{\includegraphics[width=0.3\textwidth]{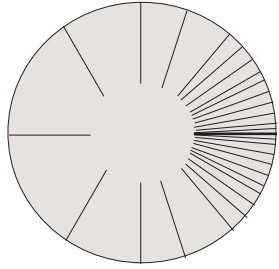}}
\caption{\label{fig-example}An example of a domain whose boundary has a non-trivial prime end. Note that if an impression of a fundamental chain
contains the point $1\in\partial W$, then it contains the whole segment $S_0$.
}
\end{figure}

We define the {\it Carath{\'e}odory compactification} $\widehat  W$ to be the disjoint union of $ W$ and  the set of prime ends
of $W$ with the following topology. For any crosscut neighborhood $N$ let $\widetilde N\subset\widehat  W$ be the neighborhood
$N$ itself, and the collection of all prime ends which can be represented by fundamental chains starting with $N$. These
neighborhoods, together with the open subsets of $ W$, form the basis for the topology of $\widehat W$.

\begin{carat}
Every conformal isomorphism $\phi: W\to\DD$ extends uniquely to a homeomorphism
$$\hat \phi:\widehat  W\to\overline\DD.$$
\end{carat}
Carath{\'e}odory Theorem for locally connected domains is a synthesis of the above statement and \propref{trivial-impression}.

Let us note:
\begin{lem}
\label{cont-lc}
If $f$ is a continuous map from a compact locally connected space $X$ onto a Hausdorff space $Y$, then $Y$ is also
locally connected.
\end{lem}
\begin{proof}
We reproduce the proof from \cite{Mil}, p. 184.
The image $f(X)=Y$ is compact. For any point $y\in Y$ and any open neighborhood $N=N(y)\subset Y$ the  set
$f^{-1}(N)$ is an open neighborhood of the compact set $f^{-1}(y)\subset X$. Consider the family $V_\alpha$ of all connected
subsets of $f^{-1}(N)$, which intersect $f^{-1}(y)$. Then the union $\cup f(V_\alpha)$ is a connected subset of $N$. It is
also a neighborhood of $y$, since it contains the open neighborhood $Y\setminus f(X\setminus \cup V_\alpha)$ of $y$.
\end{proof}
So, in particular, a Jordan curve is a locally connected set.
Using Carath{\'e}odory Theorem and \lemref{cont-lc}, we see that the converse to \propref{trivial-impression} also holds:
\begin{prop}
\label{trivial-impression2}
If the impression of every prime end of $W$ is trivial then $\partial W$ is locally connected.
\end{prop}

We also note:
\begin{thm}In the case when $ W$ is Jordan, the identity map $ W\to W$ extends to a
homeomorphism between the Carath{\'e}odory closure $\widehat W$ and $\overline W$.
\end{thm}
\begin{proof}
Suppose that there is a point $x\in\partial W$ which is in the impression of two different prime ends $p_1$, $p_2$.
Let $N_1$, $N_2$ be two
disjoint crosscut neighborhoods of $p_1$, $p_2$ respectively in $\hat W$. By Carath{\'e}odory Theorem, there exist
continuous paths $$\gamma_i:[0,1)\to\overset{\circ}{N}_i\text{ such that }\lim_{t\to 1-}\gamma_i(t)=x.$$
Let $\gamma$ be a simple curve in $W$ which is disjoint from $\gamma_i((0,1))$ and connects $\gamma_1(0)$ to $\gamma_2(0)$.
Then $$\Gamma\equiv \gamma_1([0,1))\cup\gamma_2([0,1))\cup \gamma\cup\{ x\}$$
is a Jordan curve which is easily seen to disconnect $\partial W$. Since $$\Gamma\cap \partial W=\{ x\},$$
removing a single point $x$ disconnects $\partial W$. Hence $\partial W$ is not a Jordan curve, which contradicts our assumptions.

We define a map $\iota:\widehat W\to\ov W$ by
\begin{itemize}
\item $\iota(x)=x$ for $x\in W$;
\item $\iota(p)=\cI(p)$ for a prime end $p$.
\end{itemize}
Continuity and surjectivity of $\iota$ are evident from its definition. Its injectivity was shown above. 
Since a continuous bijection of a compact set to a Hausdorff topological space is a homeomorphism onto the image, the proof is
completed.
\end{proof}

In the Jordan case, we will use the notation $\overline\phi$ for the extension of a conformal map to the closure of $ W$.
Of course,
$$\overline\phi=(\overline f)^{-1}.$$

Carath{\'e}odory compactification of $( W,w)$
can be seen as its metric completion for the following metric. Let $z_1$, $z_2$ be two points in
$ W$ distinct from $w$.
We will define the {\it crosscut distance} $\dist^ W_C(z_1,z_2)$ between $z_1$ and $z_2$ as the infimum of the diameters of
curves $\gamma$ in $ W$ with the following properties:
\begin{itemize}
\item $\gamma$ is either a crosscut or a simple closed curve;
\item $\gamma$ separates $z_1$, $z_2$ from $w$.
\end{itemize}

\noindent
It is easy to verify that:
\begin{thm}
The crosscut distance is a metric on $ W\setminus\{w\}$ which is locally equal to the Euclidean one.
The completion of $ W$ equipped with $\dist_C$ is homeomorphic to $\widehat W$.
\end{thm}

\subsection{Computational representation of prime ends}

\begin{defn}
We say that a curve $g:(0,1)\to\CC$ is a {\it rational polygonal curve} if:
\begin{itemize}
\item the image of $g$ is a simple curve;
\item $g$ is piecewise-linear with rational coefficients.
\end{itemize}
\end{defn}

\noindent
The following is elementary:
\begin{prop}
\label{specify-end}
Let $N_{\gamma_j}$ be a fundamental chain in a pointed simply-connected domain $( W,w)$. Then there exists
an equivalent fundamental chain $N_{\tau_j}$ such that the following holds. For every $j$ there exists a rational
polygonal curve $t_j:(0,1)\to\CC$ with
$$t_j(0.5)\in N_{\gamma_j}\setminus \overline{N_{\gamma_{j+1}}},$$
and such that $\tau_j\subset t_j([0.1,0.99])$.
Furthermore, $t_j$ can be chosen so that $$\diam t_j(0,1)\underset{j\to\infty}\longrightarrow 0.$$
\end{prop}

\noindent
We call the sequence of polygonal curves $t_j$ as described in the above proposition  a \textit{representation}
of the prime end
$p$ specified by $N_{\gamma_j}$. Since only a finite amount of information suffices to describe each rational polygonal curve $t_j$,
the sequence $t_j$ can be specified by an oracle. Namely, there exists 
an algorithm $\cA$ such that for every representation $(t_j)$ of a prime end there exists 
a function $\phi:\NN\to\NN$ such that
given access to  the values of $\phi(i),$ $i=1,\ldots,n$, the algorithm $\cA$  outputs the coefficients of the rational polygonal curves
$t_j$, $j=1,\ldots,m_n$ with $m_n\underset{n\to\infty}{\longrightarrow}\infty$. We will refer to such $\phi$ simply as {\it an oracle for }$p$.

\subsection{Structure of a computable metric space on $\widehat W$.}
Let $K\Subset\CC$. We say that $\phi$ is {\it an oracle for }$K$ if $\phi$ is a function from the natural numbers to
sets of finite sequences of triples $(x_j,y_j,r_j)$ of rational numbers with the following property.
Let $$\phi(n)=\{(x_j,y_j,r_j)\}_{j=1}^{k_n},$$
and let $B_j$ be the ball of radius $r_j$ about the point $x_j+iy_j$. Then
$$\dist_H(\bigcup_{j=1}^{k_n} \ov{ B_j}, K)<2^{-n}.$$

Let $( W,w)$ be a simply-connected pointed domain. Then the following conditional computability result holds:

\begin{thm}
\label{cond-comp-metric}
The following is true in the presence of oracles for $w$ and for $\partial W$.
 The Carath{\'e}odory completion $\widehat W$
equipped with the crosscut distance is a computable metric space, whose ideal points are rational points in $ W$.
Moreover, this space is computably compact.
\end{thm}

\begin{proof}
First observe that from  $w$ and $\partial W$, one can lower compute  the open sets $W$ and $\CC\setminus \overline{W}$.  In particular, we can for example enumerate all  the ideal points in $\text{Comp}(\CC)$ which are contained in $W$. Similarly, one can enumerate all the ideal points in $\text{Comp}(\CC)$  which are disjoint from $\overline{W}$. With this in mind, let us describe the algorithm $\cA$ which uniformly computes the crosscut distance between two ideal points $s_1$, $s_2$
with precision $2^{-k}$:
\begin{enumerate}
\item  $n:=1$;
\item compute a closed domain $L_n\subset  W$  which is an ideal point in $\text{Comp}(\CC)$
 such that $\dist_H(L_n, \overline{W})<2^{-n}$ (in order to do this, query the oracle to obtain
$S_n\supset \partial W$ which is a $2^{-(n+1)}$-approximation of the boundary $\partial W$, and let
$L_n$ be the closure of the union of bounded connected components of the complement $\CC\setminus S_n$);
\item compute a closed domain $U_n\supset  W$ which is an ideal point in $\text{Comp}(\CC)$
such that $\dist_H(U_n, \overline{W})<2^{-n}$ (this is done similarly to the previous step);
 \item if $s_1$ and $s_2$ do not lie in $L_n$, then go to step (8), else continue;
\item compute a lower bound $l_n$ on $\dist_C^{U_n} (s_1,s_2)$
such that $|l_n-\dist_C^{U_n} (s_1,s_2)|<2^{-n}$;
\item compute an upper bound $u_n$ on $\dist_C^{L_n} (s_1,s_2)$ such that $|u_n-\dist_C^{L_n} (s_1,s_2)|<2^{-n}$;
\item if $|l_n-u_n|<2^{-k}$ output $l_n$ and exit;
\item $n:=n+1$;
\item go to (2)

\end{enumerate}
The correctness of the argument follows from elementary continuity considerations.

The proof of computable compactness of $\widehat W$ is similarly straightforward, and is left to the reader.

\end{proof}

\subsection{Moduli of locally connected domains}
Suppose $\partial W$ is locally connected. The following definition is standard:
\begin{defn}[{\bf Modulus of local connectivity}]
Let $X\subset \RR^2$ is a connected set. Any strictly increasing function $m:(0,a)\to\RR$
is called a {\it modulus of local connectivity} of $X$ if
\begin{itemize}
\item for all $x,y\in X$ such that $\dist(x,y)<r<a$ there exists a connected subset $L\subset X$
containing both $x$ and $y$ with the property $\diam(L)<m(r)$;
\item $m(r)\searrow 0$ as $r\searrow 0$.
\end{itemize}
\end{defn}

\noindent
Of course, the existence of a modulus of local connectivity implies that $X$ is locally connected.
Conversely, every compact connected and locally connected set has a modulus of local connectivity.

We note that every modulus of local connectivity is also a modulus of path connectivity:
\begin{prop}
Let $m(r)$ be a modulus of local connectivity for a connected set $X\subset\RR^2$. Let $x,y\in X$
such that $\dist (x,y)<r$. Then there exists a path $\ell$ between $x$ and $y$ with diameter
at most $m(r)$.
\end{prop}
For the proof, see Proposition 2.2 of \cite{Orsay}.

\begin{defn}[{\bf Carath{\'e}odory modulus}] Let $( W,w)$ be a pointed simply-connected domain. A strictly increasing function
$\eta:(0,a)\to\RR$ is called the {\it Carath{\'e}odory modulus} if
 for every crosscut $\gamma$ with $\diam(\gamma)<r<a$ we have $\diam N_\gamma<\eta(r)$.
\end{defn}
We note:
\begin{prop}
There exists a Carath{\'e}odory modulus  $\eta(r)$ such that $\eta(r)\searrow 0$ when $r\searrow 0$
 if and only if the boundary $\partial W$ is locally connected.
\end{prop}
\begin{proof}
Let us begin by assuming that there exists a Carath{\'e}odory modulus with $\eta(r)\searrow 0$. Then the impression of every prime end of
$( W,w)$ is a single point. By Carath{\'e}odory Theorem, this implies that $\partial W$ is a continuous
image of the circle, and hence, is locally connected.

Arguing in the other direction, if $\partial W$ is locally connected, then by \propref{trivial-impression} 
every fundamental chain shrinks to a single
boundary point. The existence of a desired Carath{\'e}odory modulus follows from compactness considerations.

\end{proof}

\section{Statements of the main results}
 To simplify the exposition, we present our results for bounded domains only.
However, all the theorems we formulate below may be stated for general simply-connected domains on the 
Riemann sphere $\riem$. In this case,
the spherical metric on $\riem$ would have to be used instead of the Euclidean both in the statements and in the proofs.

Our first result is the following:
\begin{thm}
\label{main1}
Suppose $( W,w)$ is a bounded simply-connected pointed domain. Suppose the Riemann mapping
$$\phi: W\to\DD\text{ with }\phi(w)=0,\;\phi'(w)>0$$
is computable.
Then there exists an algorithm $\cA$ which, given a representation of a prime end $p\in\widehat W$
computes the value of $\hat\phi(p)\in S^1.$

\end{thm}

\noindent
In view of \thmref{comp-riem-map},  we have:
\begin{cor}
\label{cormain1}
Suppose we are given oracles for $W$ as a lower-computable open set, and for $\partial W$ as a lower-computable closed set,
and an oracle for the value of $w$ as well.
Given a representation of a prime end $p\in\widehat W$, the value $\hat\phi(p)\in S^1$ is uniformly computable.
\end{cor}

\noindent
To state a ``global'' version of the above computability result, we use the structure of a computable metric space:
\begin{thm}
\label{main2}
In the presence of oracles for $w$ and for $\partial W$, both the Carath{\'e}odory extension
$$\hat\phi:\widehat W\to\overline\DD\text{ and its inverse }\hat f\equiv \hat\phi^{-1}:\overline\DD\to\widehat W$$
are computable, as functions between computable metric spaces.
\end{thm}

\begin{rem}
Note that the assumptions of \thmref{main2} are stronger than those of \corref{cormain1}: computability of $\partial W$
implies lower computability of $W$ and $\partial W$, but not vice versa.
\end{rem}
\begin{proof}
The assertion that computability of $\partial W$ implies lower computability of $W$ and $\partial W$ is straightforward from the
definitions.

To show that the converse statement does not hold consider the following example.
Let $B\subset \NN$ be a lower-computable, non-computable set, and let $\cA$ be an algorithm which enumerates the elements of $B$.
Let $I$ be the boundary of the unit square $I\equiv \partial([0,1]\times[0,1])$.
Set $x_{i}=1-1/{2i}$. We now modify $I$  about the points $x_{i}$.   If $i\notin B$, we do nothing.
If $i\in B$ and it is enumerated by $\cA$ at step $s$, then remove the segment
from $(x_{i}-s_{i},1)$ to $(x_{i}+s_{i},1)$ where
$$s_{i}=\min\{2^{-s}, 1/(3i^{2})\},$$ and add  straight segments connecting $(x_{i}-s_{i},1)$
to $(x_{i}-s_{i}, 1.5)$  to $(x_{i}+s_{i}, 1.5)$ to $(x_{i}+s_{i},1)$ (we call such a decoration an $i$-peninsula of width $s_i$).
Denote $S$ thus obtained set and let
$W$ be the bounded connected component of $\CC\setminus S$.

Note that computability of $\partial W=S$ would imply computability of $B$: to check whether $i\in B$ it is sufficient to see
whether $S$ has an $i$-peninsula, which is equivalent to $(x_i,1.5)\in S$. Hence, $\partial W$ is not computable.

However, $\partial W$ is lower computable. Indeed, to lower compute $\partial W$, start by drawing $I$ with the segments $(x_{i}-1/(3i^{2}),x_{i}+1/(3i^{2}))$ removed. Then emulate the algorithm $\cA$ enumerating $B$ and at each step $s$:
\begin{itemize}
\item if $i$ is enumerated, then draw the corresponding $i$-fjord about $x_{i}$.
\item For every $j<s$ that has not been enumerated so far,  narrow the removed segment  $(x_{j}-1/(3i^{2}),x_{j}+1/(3i^{2}))$ to $(x_{j}-s_{i},x_{j}+s_{i})$.
\end{itemize}
Similarly, the domain $W$ is lower computable. The procedure to lower compute it is the following: start by drawing the unit
square $[0,1]\times [0,1]$. Then emulate the algorithm $\cA$ enumerating $B$ and at each step $s$:
\begin{itemize}
\item if $i$ is enumerated, add the rectangle bounded by straight segments from
$(x_{i}-s_{i},1)$ to
to $(x_{i}-s_{i}, 1.5)$  to $(x_{i}+s_{i}, 1.5)$ to $(x_{i}+s_{i},1)$  to $(x_{i}-s_{i},1)$.
\end{itemize}

\end{proof}

\noindent
For a domain $ W$ with a locally connected boundary, we may ask when the map $\overline f:\overline\DD\to\overline W$ is computable.
We are able to give a sharp result:
\begin{thm}
\label{main3}
Suppose $( W,w)$ is a pointed simply-connected bounded domain with a locally connected boundary. Assume that  the Riemann map
$$f:\DD\to W\text{ with }f(0)=w,\; f'(0)>0$$
is computable.

Then the boundary extension
$$\overline f:\overline\DD\to\overline W$$
is computable if and only if there exists a computable Carath{\'e}odory modulus $\eta(r)$ with
$\eta(r)\searrow 0$ as $r\searrow 0$.
\end{thm}

\begin{remark}
With routine modifications, the above result can be made \emph{uniform} in the sense that there is an algorithm which from a description of $f$ and $\eta$ computes a description of $\overline f$, and there is an algorithm which from a description of $\overline f$
computes a Carath{\'e}odory modulus $\eta$. See for example \cite{Her} for statements made in this generality.
\end{remark}

We note that the seemingly more ``exotic'' Carath{\'e}odory modulus cannot be replaced by the modulus of local connectivity in the above
statement:
\begin{thm}[{\bf Computational Incommensurability of Moduli }]
\label{main4}
There exists a simply-connected domain $ W$ such that $\partial W$ is  locally-connected, $\overline W$ is computable,
and there exists a computable Carath{\'e}odory modulus $\eta(r)$, however, no computable modulus of local connectivity exists
for $\partial W$.
\end{thm}
Finally, we turn to computational complexity questions in the cases when $\hat\phi$ or $\overline f$ are computable. We show that:
\begin{thm}
\label{main5}
Let $q:\NN\to\NN$ be any computable function. There exist Jordan domains $ W_1\ni 0$, $ W_2\ni 0$ such that the following holds:
\begin{itemize}
\item the closures $\overline W_1$, $\overline W_2$ are computable;
\item the extensions $\overline\phi:\overline W_1\to\overline\DD$ and  $\overline f:\overline \DD\to\overline W_2$  are both computable functions;
\item the time complexity of $\overline f$ and $\overline\phi$ is bounded from below by $q(n)$ for large enough values of $n$.
\end{itemize}

\end{thm}

\section{Preliminaries from Complex Analysis}
\label{sec:complex}
\subsection{Distortion Theorems}
We will make use of two standard results in Complex Analysis.

\begin{onequarter}
Let $f:\DD\to W$ be a conformal isomorphsim, and let $ w=f(0)$. Then the distance from $ w$ to the
boundary of $U$ is at least $\frac{1}{4}|f'(0)|$.
\end{onequarter}

\begin{distortion}
Let $f:\DD\to\DD$ be a conformal mapping with the properties $f(0)=0$ and $f'(0)=1$. Then for every $z\in\DD$,
$$ \frac{1-|z|}{(1+|z|^3)}\leq |f'(z)|\leq \frac{1+|z|}{(1-|z|^3)}.   $$
\end{distortion}

\subsection{Harmonic Measure}
A detailed discussion of harmonic measure can be found in \cite{Garnett-Marshall}. Here we briefly recall some of the relevant facts. We will only define harmonic measure for a finitely-connected domain $W\subset\widehat\CC$. Recall that a connected subdomain of $\riem$ is called {\it hyperbolic} if its complement contains at least three points. We start with the following well-known fact:
\begin{prop}
Let $W$ be a finitely-connected hyperbolic subdomain of $\widehat\CC$, and $w\in W$. Then the Brownian path originating at $w$ hits the boundary of $W$ with probability $1$.
\end{prop}
Let $W$ be a finitely-connected hyperbolic domain  in $\riem$, and $w\in W$. The {\it harmonic measure}
$\omega_{W,w}$ is defined on the boundary $\partial W$. For a set $E\subset \partial W$ it is equal to the probability that
 the Brownian path originating at $w$ will first hit $\partial W$ within the set $E$.

By way of an example, consider a simply-connected hyperbolic domain $W\subset \riem$ with locally-connected boundary, let $w$ be an arbitrary point of $W$.
Consider the unique conformal Riemann mapping
$$\psi:W\to\DD,\text{ with }\psi(w)=0\text{ and }\psi'(w)>0.$$

By Carath{\'e}odory Theorem, $\psi^{-1}$ extends continuously to map $\overline W\to \bar\DD$. By symmetry considerations,
the harmonic measure $\omega_{\DD,0}$ coincides with the Lebesgue measure $\mu$ on the unit circle $S^1=\partial\DD$. Conformal invariance of Brownian motion
implies that $\omega_{W,w}$ is obtained by pushing forward $\mu$ by $\psi^{-1}|_{S^1}$.

We will repeatedly use the following estimate on the harmonic measure

\begin{prop}[Majoration principle]\label{majoration}
Let $W'\subset W$ be two finitely-connected hyperbolic domains in $\riem$. Let $w\in W'$. Let $$K_1=\partial W\cap\partial W',\; K_2=\partial W\setminus K_1,\; K_3=\partial W'\setminus K_1.$$
 Then
$$\omega_{W,w}(K_2)\leq\omega_{W',w}(K_3).$$
\end{prop}

\begin{proof}
Evidently, the Brownian path originating at $w$ which exits $W$ through $K_2$ must exit $W'$ through $K_3$. The statement now follows from the definition of harmonic measure.
\end{proof}
The proof of the following classical result can be found in \cite{Garnett-Marshall}:
\begin{bpt}
If $K\subset \bar\DD\setminus\{ 0\}$, and $K^*\equiv \{|w|\;:\; w\in K\}$ is the circular projection of $K$, then for every $z\in\DD\setminus K$
$$\omega_{\DD\setminus K, z}(K)\geq \omega_{\DD\setminus K^*,-|z|}(K^*).$$
\end{bpt}

\subsection{Estimating the variation of $f:\DD\to W$ (Warshawski's Theorems).}
We will quote several results from the beatifully concise paper of S.E.~Warshawski \cite{War}.
Let us consider a conformal map $f:\DD\to W$. Without assuming that $ W$ is necessarily Jordan,
we can make the following definition.

\begin{defn}
Let $|z_0|=1.$ For $0<r<1$, define
$$\cO(r;z_0)=\underset{|z_k-z_0|\leq r}{\sup}|f(z_1)-f(z_2)|\text{ where }|z_1|<1,\;|z_2|<1,$$
and
$$\cO(r)=\underset{|z_0|=1}{\sup}\cO(r;z_0).$$
The quantity $\cO(r)$ is called the {\it oscillation of }$f(z)$ {\it at the boundary.}
\end{defn}

The first theorem in \cite{War} is the following:
\begin{thm}
\label{war1}
Suppose $ W\ni 0$ is a simply-connected bounded region, and $f:\DD\to  W$ is a conformal mapping such that $f(0)=0$.
Assume that $\eta(r)$ is a Carath{\'e}odory modulus of $( W,0)$.
If $A$ denotes the area of $ W$, then the oscillation of $f(z)$ at the boundary
$$\cO(r)\leq \eta\left(\left(\frac{2\pi A}{\log 1/r}\right)^{1/2}\right),\text{ for all }r\in(0,1).$$

\end{thm}

\noindent
The proof follows easily from Wolff's Lemma \cite{Wolff}:
\begin{lem}
\label{wolff}
Suppose that $f$ is a conformal mapping of $\DD$ onto a simply-connected bounded region $ W$. Let $z_0\in S^1$, and let $k_r$
be the arc of the circle $|z-z_0|=r$ which is contained in $\DD$. Then for every $r\in(0,1)$ there exists $\rho^*\in(r,\sqrt{r})$
such that the image of $k_{\rho^*}$ is a crosscut of $ W$ of length
\begin{equation}
\label{eq:wolff}
\ell_{\rho^*}\leq\left(\frac{2\pi A}{\log 1/r}\right)^{1/2}.
\end{equation}
\end{lem}
\begin{proof}
We introduce polar coordinates about $z_0$ and write, for $\rho\in(0,1),$
$$\ell_\rho=\int_{k_\rho}|f'(z)||dz|=\int_{k_\rho}|f'(z_0+\rho e^{i\theta})|\rho d\theta\leq +\infty.$$
By Cauchy-Schwarz-Buniakowsky Inequality,
$$\ell_\rho^2\leq \int_{k_\rho}|f'(z_0+\rho e^{i\theta})|^2\rho d\theta\int_{k_\rho}\rho d\theta\leq \pi\rho \int_{k_\rho}|f'(z+\rho e^{i}\theta)|^2\rho d\theta.$$
Integrating with respect to $\rho$ from $r$ to $\sqrt{r}$, we obtain
$$\int_r^{\sqrt{r}}\frac{\ell_\rho^2}{\rho}d\rho<\int_0^{\sqrt{r}}\frac{\ell_\rho^2}{\rho}d\rho\leq \pi\int_0^{\sqrt{r}}\int_{k_\rho}|f'(z_0+\rho e^{i\theta})|^2\rho d\theta<\pi A.$$
Hence, there exists $\rho^*\in(r,\sqrt{r}), $ such that
$$\ell_{\rho^*}^2\int_r^{\sqrt{r}}
\frac{d\rho}{\rho}=
\frac{1}{2}\ell_{\rho^*}^2
\log\frac{1}{r}<\pi A.$$
Since the image of $k_{\rho^*}$ has a finite length, it is a crosscut in $ W$, and the proof is completed.

\end{proof}

\begin{proof}[Proof of \thmref{war1}]
Let
$$T_r\equiv f(\{|z-z_0|<r,\;|z|<1\})\subset W.$$
Select $\rho^*\in(r,\sqrt{r})$ so that (\ref{eq:wolff}) holds. Then $T_r\subset T_{\rho^*}$. The image
$$\gamma\equiv f(k_{\rho^*})$$
is a crosscut of $ W$ with $\diam\gamma\leq\ell_{\rho^*}.$ Hence, by definition of a Carath{\'e}odory modulus,
$$\diam T_r\leq \diam T_{\rho^*}\leq  \eta\left(\left(\frac{2\pi A}{\log 1/r}\right)^{1/2}\right).$$
If $z_1,z_2\in\DD\cap \{|z-z_0|<r\},$ then $f(z_1)$, $f(z_2)\in T_r$, and the proof is completed.
\end{proof}

We quote another theorem of \cite{War} without a proof. First, we make a definition:
\begin{defn}

Let $ W_1$ and $ W_2$ be two simply-connected regions. Let us define the {\it inner distance} between $ W_1$, $ W_2$
as
$$\dist_i( W_1, W_2)\equiv \max(d(\partial W_1\cap\overline W_2,\partial W_2),d(\partial W_2\cap\overline W_1,\partial W_1),$$
where, as usual, $d(X,Y)$ denotes the ``one-sided'' distance
$$d(X,Y)=\sup_{x\in X}\inf_{y\in Y}\dist(x,y).$$
\end{defn}

\begin{thm}
\label{war2}
Suppose $ W_i\ni 0$ for $i=1,2$ are two simply-connected bounded regions.
Let $\sigma=\dist(0,\partial W_1\cup\partial W_2).$
Suppose $\dist_i( W_1, W_2)<\eps$ where $\eps\in(0,1)$ and $\eps<\sigma/64.$ Let $\eta_i(r)$ denote the Carath{\'e}odory modulus
of $ W_i$. Let $A_i$ be the area of $ W_i$.
Denote $f_i(z)$ the conformal map $\DD\to W_i$ with $f_i(0)=0$, $f_i'(0)>0$. Then for $z\in\DD$ we have
$$|f_1(z)-f_2(z)|\leq\left( 1+\frac{k}{\sigma^{1/2}}\eps^{1/4}\log\frac{4}{\eps} \right)\left[\eta_1\left(\left(\frac{8\pi A_1}{\log(1/\eps)}\right)^{1/2}\right)+\eta_2\left(\left(\frac{8\pi A_2}{\log(1/\eps)}\right)^{1/2}\right)\right],$$
where $k$ is a constant $k\leq 16e$.
\end{thm}

\section{Proofs}

\subsection{Proof of \thmref{main1}}

We will show that from a computable description $\{\gamma_{i}\}$ of a  prime end $p$, we can compute a description of the point $z=\hat{\varphi}(p)\in \partial{D}$.
We will describe an algorithm that,  given $\{A_{i}\}$ and $\epsilon$, will find $z'\in D$ such that $|z-z'|\leq \epsilon$. The following proposition will give us the key estimate.

\begin{proposition}\label{Lavrentiev} Let $W$ be a connected domain with $\infty\not\in W$. Let $\gamma$ be a crosscut of $ W$, $\dist(\gamma,  w)\geq M$, for some $M>0$ and $ N_\gamma$ be the component of $ W\setminus\gamma$ not containing $ w$. Assume that $\epsilon^2<M/4$. Then
$$\diam(\gamma)\leq \epsilon^{2} \implies \diam(\varphi( N_\gamma))\leq \frac{30\epsilon}{\sqrt{M}}.$$
\end{proposition}
Similar statements are well-known in the literature, likely beginning with the works of Lavrientieff \cite{Lav} and Ferrand \cite{Fer}.
\noindent
As an immediate corollary, we obtain
\begin{corollary}\label{ContinuityToCaratheodory}
Let $h(\delta)$ be a  modulus of fluctuation of $f=\varphi^{-1}$. Then we have the following estimate for Caratheodory modulus:
$$\eta(\epsilon^2)\leq h\left(\frac{30\epsilon}{\sqrt{M}}\right).$$
\end{corollary}
\begin{proof}
Let two points $w_1$ and $w_2$ be separated from $ w$ by a crosscut of length at most $\epsilon^2$. Let $z_{1,2}=\phi(w_{1,2})$. By Proposition \ref{Lavrentiev}, $|z_1-z_2|\leq\frac{30\epsilon}{\sqrt{M}}$. Thus
$$|w_1-w_2|=|f(z_1)-f(z_2)|\leq\frac{30\epsilon}{\sqrt{M}}.$$
\end{proof}

Let us now to turn to the proof of Proposition \ref{Lavrentiev}. We use the following lemmas.
\begin{lemma}\label{1} If  $\diam(\gamma)<\epsilon^{2}$ then for any $z\in\varphi(\gamma)$,  $$|z|>1-\frac{3\epsilon}{\sqrt{M}}.$$
\end{lemma}

\begin{proof}
Let $f=\varphi^{-1}$ and let $ w=f(0)\in  W$. By Koebe One-Quarter Theorem, $|f'(0)|\geq 4 M$.
By Koebe Distortion Theorem,
$$|f'(z)|\geq 4 M\frac{1-|z|}{(1+|z|)^3}\geq M\frac{1-|z|}{2}.$$
 Another application of Koebe One-Quarter Theorem gives
$$\dist(f(z), \partial  W)\geq\frac14(1-|z|)|f'(z)|\geq M \frac{(1-|z|)^2}{8}.$$

Notice now that since $\gamma$ is a crosscut, $$\diam(\gamma)\geq\dist(f(z), \partial  W)\geq M \frac{(1-|z|)^2}{8}.$$
The lemma immediately follows from the last inequality.

\end{proof}

\begin{lemma}\label{2}
Let $\gamma$ be a crosscut of $ W$ and let $M=\dist(\gamma,  w)$. Let $$\diam(\gamma)\leq \epsilon^2<M/4.$$ Suppose $\varphi(\gamma)\subset \bar{D}$ ends at $z_{1}$ and $z_{2}$, then $$|z_{1}-z_{2}|<\frac{8\epsilon}{\sqrt M}.$$

In addition, if $K$ is the part of $\partial  W$ separated by $\gamma$ from $ w$, then
$$\omega_{ W, w}(K)\leq\frac{4\epsilon}{\pi\sqrt{M}}.$$
\end{lemma}
\begin{proof}
First observe that by conformal invariance of harmonic measure
$$ \text{length of arc}[z_1,z_2]=2\pi\omega_{ W, w}( K).$$
Thus $$|z_1-z_2|\leq2\pi\omega_{ W, w}( K),$$ so the second statement of the Lemma implies the first one.

Let $x$ be any point of $\gamma$, and
let $ W'$ be the component of $ W\setminus B(x, \epsilon^2)$ containing $w$. Let $K_1=\partial  W'\cap  \partial B(x,\epsilon^2)$ and $K_2=\partial  W'\setminus K_1$.  Since $\diam\gamma\leq\epsilon^2$, we see that $W'$ is simply-connected.  Thus we can apply the Majoration Principle (Proposition \ref{majoration}) to show that
\begin{equation}\label{eq:proj1}\omega_{ W, w}( K)\leq \omega_{ W', w} (K_1).\end{equation}

\begin{figure}
\centerline{\includegraphics[width=\textwidth]{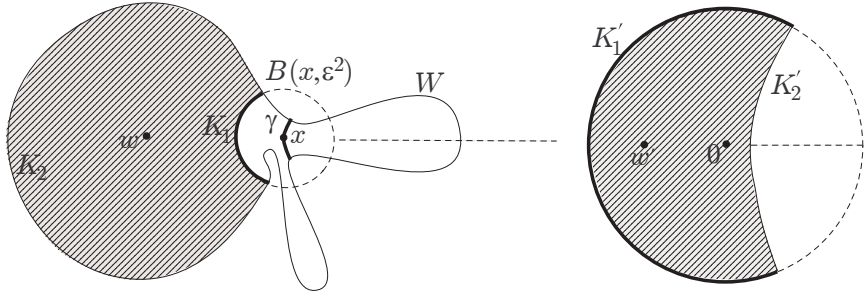}}
\caption{\label{fig-beurling}An illustration to the proof of \thmref{main1}. The domain $W'$ is the filled-in subdomain of $W$. For simplicity, $K_1$ is shown to be connected.
}
\end{figure}
We will use Beurling Projection Theorem to obtain an upper bound on $\omega_{ W', w} (K_1)$. By applying a shift by $-x$ and a rotation, we may assume that $x=0$ and that $w$ lies on the negative real axis.

Let us consider the inversion map
$g(w)=\frac{\epsilon^2}{w}$. It maps $W'$ to a subdomain $g(W')\subset\DD$. We set $$w'=g(w),\; K'_1=g(K_1)=\partial g(W')\cap S^1,\; K'_2=g(K_2)=\partial g(W')\cap \DD.$$

Let $K^*_2$ be the circular projection of $K_2'$. Since by our assumptions $w'=-|w'|$,  Beurling Projection Theorem implies that

$$\omega_{g(W'), w'}(K'_2)\geq \omega_{\DD\setminus K^*_2, w'}(K^*_2),$$
or, since $\omega_{g(W'), w'}(K'_2)+\omega_{g(W'), w'}(K'_1)=1$ and $\omega_{\DD\setminus K^*_2, w'}(K^*_2)+\omega_{\DD\setminus K^*_2, w'}(S^1)=1$,
$$\omega_{g(W'), w'}(K'_1)\leq \omega_{\DD\setminus K^*_2, w'}(S^1)$$
Now we can use the conformal invariance of harmonic measure to see that
$$\omega_{W', w}(K_1)\leq  \omega_{g^{-1}\left(\DD\setminus K^*_2\right), w}(\partial B(x,\epsilon^2))$$

 Another application of the Majoration Principle (Proposition \ref{majoration}) shows that
 $$\omega_{W', w}(K_1)\leq\omega_{w, W_0}(\partial B(x,\epsilon^2)),$$
 where $W_0=\riem\setminus (B(x,\epsilon^2)\cup [0,\infty))$.

The last quantity can be shown to be equal to
$$\frac{4}{\pi}\tan^{-1} \left(\frac{\epsilon}{\sqrt{M}}\right)\leq\frac{4\epsilon}{\pi\sqrt{M}}$$ (see \cite{Garnett-Marshall}, page 107).
 The estimate  together with inequality \eqref{eq:proj1} immediately implies the second statement of the lemma.
\end{proof}

\begin{proof}[Proof of the Proposition \ref{Lavrentiev}] 
We first show that the radial projection of $\varphi(\gamma)$ onto the unit circle has the length bounded by $\frac{25\epsilon}{\sqrt{M}}$. Let this projection be the arc $[z'_1, z'_2]$.

Let us first estimate $|z'_1-z_1|$. Without loss of generality we can assume that $z'_1=1$. Let
$$r=\max\left\{x>0\ :\ x\in\varphi(\gamma)\right\},\; w=\varphi^{-1}(r)=f(r)\text{  and }I=[r,1].$$ Let $\gamma'$ be the  arc of $\gamma$ joining $w$ to the end of $\gamma$ corresponding to $z_1$. Then $\gamma'$ is a crosscut of the domain $ W'=  W\setminus f(I)$ with $$\dist( w, \partial  W')\geq M-\epsilon^2\geq M/2.$$

Let $K'$ be the part of the boundary of $ W'$ separated by $\gamma'$ from $ w$. By Lemma \ref{2},
 $$\omega_{ W', w}( K')\leq\frac{4\sqrt{2}\epsilon}{\pi\sqrt{M}}.$$ Notice now that by invariance of harmonic measure and symmetry we have
$$\omega_{ W', w}(K')=\omega_{\DD\setminus I,0}(\phi(K))=\frac{1}{2}\omega_{\DD\setminus I,0}(I\cup[z_1,\overline{z_1}]), $$
where $[z_1,\overline{z_1}]$ denote here the arc of the unit circle joining $z_1$ and $\overline{z_1}$.

Another application of majoration principle shows that

$$|z_1-z'_1|\leq1/2 \text{ length of }[z_1,\overline{z_1}]=\pi\omega_{\DD,0}(I\cup[z_1,\overline{z_1}])\leq\pi\omega_{\DD\setminus I,0}(I\cup[z_1,\overline{z_1}])\leq \frac{12\epsilon}{\sqrt{M}}.$$

Since the same estimate holds for $|z_1-z'_2|$, we get the desired estimate on the length of the projection.

To obtain the statement of the proposition, we just need to combine this estimate with Lemma \ref{1}.

\end{proof}
We are now ready to prove \thmref{main1}.
\begin{proof}[Proof of \thmref{main1}]
Consider a computable representation $(t_j)$ of the crosscut $p$. To compute $\phi(p)$ with presicion $\delta$, we compute $j$ such that $\diam t_j<10^{-3}M\delta^2$, where $M$ is a lower estimate on the distance from $t_1$ to $w$. Then we compute $\phi(t_j(1/2))$ with presicion $10^{-3}\delta$. By Proposition \ref{Lavrentiev}, this value is within $\delta$ of $\phi(p)$. 
\end{proof}

\subsection{Proof of \thmref{main2}}
Let us show that $\hat\phi$ is computable, computability of $\hat f$ will follow from \thmref{invert}.

Note that given oracles for $\partial\Omega$ and $w$, the Caratheodory distance between two interior points is computable by using, say, computable interior polinomial approximation.

Let now $x_n\to x$ be a sequence of ideal points in $\widehat W$ with $\dist_C^W(x_n,x)<2^{-n}$. We compute a rational number
$M>0$, which is a lower bound
on $\dist(x, w)$. We set
$$\eps_k=\frac{\sqrt{M}2^{-{2k}}}{60},$$
and let $ n_k\geq 2k-\log_2\frac{\sqrt{M}}{60}$ be a natural number.
Set $y_k\equiv x_{n_k}$, and compute $z_k\in\DD$ such that
$$|z_k-\phi(y_k)|<\eps_k.$$
By \propref{Lavrentiev},
$$|z_k-\hat\phi(x)|<2^{-n}.$$

\hspace{0.9\textwidth}$\Box$

\subsection{Proof of \thmref{main3}}


By \thmref{war1} of Warshawski, we know that
$$
\cO(r)\leq \eta\left( \left(\frac{2\pi A}{\log 1/r}  \right)^{1/2}\right)
$$
where $A$ is an upper bound on the area of the domain $U$.  Assume  $f$ and $\eta$ are computable.

 Let $z\in \bar{U}$ and $\varepsilon>0$.  We now show how to compute $\bar{f}(z)$ at precision $\varepsilon$.  Start by choosing a computable but not rational number $r$ such that
 $$
 \eta\left( \left(\frac{2\pi A}{\log 1/(2r)}  \right)^{1/2}\right) < \frac{\varepsilon}{2}.
 $$

Then take a rational  approximation $z'$ of $z$ such that $|z-z'|<r/10$.  Since $|z'|\neq 1-r$, we can decide whether $|z'|>1-r $ or $|z'|<1-r$.  If $|z'|>1-r $, then compute $f(z')$ at precision $\varepsilon/2$. By \thmref{war1}, this is an $\varepsilon$ approximation of $\bar{f}(z)$.  If $|z'|<1-r$, then $z\in D$, and therefore we can just compute $f(z)$.

For the converse, assume  $\bar{f}$ is computable. Since $\bar{D}$ is computably compact, so is the set $\{|z-z'|\leq\delta\} \subset \bar{D}\times \bar{D}$ and therefore we can compute, by Proposition \ref{modcomp},  the modulus of fluctuation of $\bar{f}$:
 $$h(\delta):=\sup_{|z,z'|\leq \delta}|\bar{f}(z)-\bar{f}(z')|.$$

 Computability of the rate decay of  $\eta(\delta)$ now follows from Corollary \ref{ContinuityToCaratheodory}.

\hspace{0.9\textwidth}$\Box$

\begin{rem}
An alternative proof of computability of $\bar f$ in \thmref{main3} can be obtained using the second theorem of Warshawski: we can compute a sequence of polygons approximating $\partial W$ with precision $2^{-n}$ in Hausdorff sense and with the same Carath{\'e}odory modulus and then apply the estimate from \thmref{war2}.
\end{rem}

\subsection{Proof of \thmref{main4}}
Let $B\subset \NN$ be a lower-computable, non-computable set.
Let $I$ be the square $[0,1]\times[0,1]$.
Set $x_{i}=1-1/{2i}$. The boundary $\partial W$ is constructed by modifying $I$ as follows.  If $i\notin B$,
then we add a straight line to $I$ going  from $(x_{i},1)$ to $(x_{i},x_{i})$.
We call these  {\it $i$-lines}.  If $i\in B$ and it is enumerated in stage $s$, then remove the segment
from $(x_{i}-s_{i},1)$ to $(x_{i}+s_{i},1)$ where
$$s_{i}=\min\{2^{-s}, 1/(3i^{2})\}.$$ To close the domain, we join by straight lines $(x_{i}-s_{i},1)$
to $(x_{i}-s_{i},x_{i})$  to $(x_{i}+s_{i},x_{i})$ to $(x_{i}+s_{i},1)$. Call these {\it $i$-fjords}.
 This completes the construction of $\partial W$ (see Figure \ref{fig-square}).

We now show how to compute a $2^{-s}$ Hausdorff approximation of the boundary. Start by running an algorithm $\cA$ enumerating
 $B$ for $s+1$ steps. For all those $i$'s that have been enumerated so far, draw the corresponding $i$-fjords.
For all the other $i$'s, draw a $i$-line.
This is clearly a $2^{-s}$ approximation of $W$ since for any $i$ enumerated after the $s+1$ steps, the Hausdorff distance between the $i$-line and the $i$-fjord is less than $2^{-s}$. There clearly exists a computable Carath{\'e}odory modulus.
For example we can take $$\eta(r)=2\sqrt r \text{ for }r<1.$$

\begin{figure}
\centerline{\includegraphics[height=0.3\textheight]{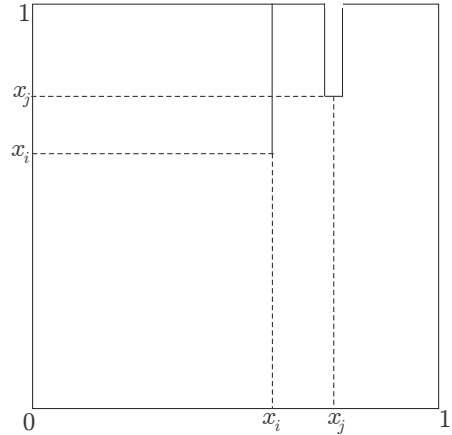}}
\caption{\label{fig-square}An illustration to the proof of \thmref{main4}: a square with an $i$-line and a $j$-fjord.
}
\end{figure}

Assume that the modulus of local
connectivity $m(r)$ is also computable. We then arrive at a contradiction by showing that $B$ is a computable set. First,
using monotonicity of $m(r)$, we can, for every value of $i\in\NN$, compute
$r_i\in \QQ$ such that
$$m(2\cdot 2^{-r_i})<\frac{x_i}{2}.$$ It then follows, that
if $i\in B$ then $i$ is enumerated by $\cA$ in fewer than $r_i$ steps. Our algorithm to compute $B$ will emulate  $\cA$
for $r_i$ steps to decide whether $i$ is
an element of $B$ or not.

\hspace{0.9\textwidth}$\Box$

\subsection{Proof of \thmref{main5}}
To begin constructing the Jordan domain $W_1$, let us choose a sequence of points $w_n\in S^1$ such that:
\begin{itemize}
\item $w_n=\exp(i\theta_n)$ with $\theta_n\searrow 0$;
\item $|\theta_n-\theta_{n+1}|=2^{-(n+q(n))}$.
\end{itemize}
For $w\in S^1$ and $r\in(0,0.5)$ let us define a {\it wedge}
$$Q(w,\eps,r)\equiv\{(1-s)\exp\left(i \left(1-\frac{s}{r}\right)\theta\right)\text{ where }s\in[0,r],\; |\theta-\arg(w)|\leq\eps \}.$$
Let us also define an auxiliary domain
$$\DD_{r}\equiv\DD\setminus\left\{z=s\exp(i\theta)\text{ where }s>1-r,\ |\theta|<r\right\},$$
where again $r\in(0,0.5)$.

Let $z_n\in S^1$ be the midpoint of the arc $[w_n,w_{n+1}]$.

We claim that we can compute
a subsequence $z_{n_k}$, and a sequence $r_k\searrow 0$ such that denoting  $$Q_{k}\equiv Q(z_{n_k},(\theta_{n_k}-\theta_{n_{k+1}})/4,r_k),$$ we have:
\begin{itemize}
\item $r_k<\frac{1}{k}$;
\item $Q_k\subset \overline{\DD_{r_k}}$;
\item let $W^k$ be the connected component of $0$ in $\DD_{r_k}\setminus (\cup_{j=1}^{k-1} Q_j)$ (see Figure \ref{fig-wedges}). Then the harmonic measure
\begin{equation}\label{main5:1}\omega_{W^k, 0}(Q_j )\geq 2^{-n_j}\text{ for all }j\leq k-1\end{equation}

\end{itemize}
Let us argue inductively. For the initial step we fix $n_1=1$. By continuity considerations there exists $r_1$ such that
the harmonic measure
$$\omega_{\DD\setminus Q_1,0}(Q_1)>2^{-1}. $$
Again by continuity, there exists $r_2$ such that (\ref{main5:1}) holds for $k=2$. Now such a pair $Q_1$, $r_2$
 can be computed by \thmref{main2} using an exhaustive search.

Let us assume that $z_{n_1},\ldots, z_{n_{k-1}}$ and $r_1,\ldots,r_k$ have already been constructed.
Set $$\delta_k\equiv \exp\left( -\frac{\pi^2}{2r_k^2}\right).$$
Suppose, $Q=Q(w,\eps,r_k)\subset \CC\setminus W^k$.
Set $\Omega_m\equiv \DD\setminus (\cup_{j=1}^{m} Q_j)$.
By Wolff's Lemma \eqref{eq:wolff},
\begin{equation}
\label{main5:2}
\omega_{\Omega_{k-1}\setminus Q,0}(Q)>\delta_k.
\end{equation}
Select $n_k$ so that
\begin{itemize}
\item $w_{n_k}\notin W^{k-1}$, and
\item $\delta_k>2^{-n_k}$.
\end{itemize}
Then, by \eqref{main5:2},
$$ \omega_{\Omega_{k},0}(Q_k)>\delta_k>2{-n_k}.$$
By continuity considerations, there exists $r_{k+1}<\frac{1}{k+1}$ such that
$$ \omega_{\Omega_{k}\cap\DD_{r_{k+1}},0}(Q_k)>2{-n_k}.$$
By \thmref{main2}, such $r_{k+1}$ can be computed, using an exhaustive search.

By Majoration Principle \ref{majoration},
$$\omega_{W^k, 0}(Q_j )\geq 2^{-n_j}\text{ for all }j\leq k ,$$
and the induction is completed.

Let $$W_1\equiv \cup W^k=\DD\setminus(\cup_{j=1}^\infty Q_j). $$
Denote $\phi$ the normalized conformal map $W_1\to\DD$ with $\phi(0)=0,$ $\phi'(0)>0$. Set $\phi(w_k)=\exp (2\pi i \zeta_k)$.
Using Majoration Principle again,
$$|\phi(w_{k+1})-\phi(w_k)|>\dist_{\RR/\ZZ}(\zeta_{n_k},\zeta_{n_{k+1}})> 2^{-n_k}.$$

\begin{figure}
\centerline{\includegraphics[height=0.3\textheight]{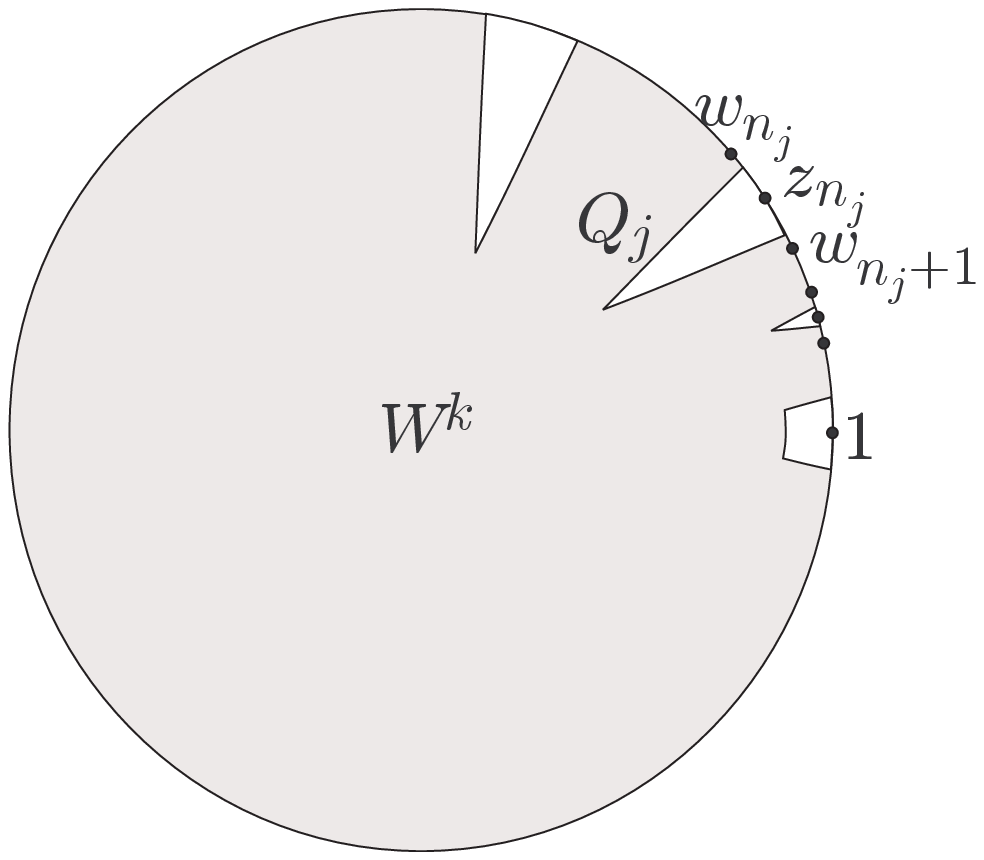}}
\caption{\label{fig-wedges}An illustration to the proof of \thmref{main5}.
}
\end{figure}

Observe that  at least $q(n)+1$ bits are needed to separate $w_{n}$ from $w_{n+1}$, and therefore $q(n)+1$ computer steps. Thus, to compute the function $\phi$ with
precision $2^{-n_k}$, we need the time of at least $q(n_{k})+1$, and the proof of the first half of the theorem is finished.

For the benefit of a reader without prior experience with similar Computability Theory arguments, let us informally summarize the above
proof as follows: to estimate the value of the conformal mapping $\phi$ for the domain $W_1$ constructed above with precision $2^{-n}$ the
computation has to be carried out with a very high precision (with at least $2^{-(q(n)+1)}$ dyadic digits). Both in theory and in computing practice,
such computations are costly, and, in particular, would require processing time of at least $O(q(n))$.

For the second part of the theorem, note that by
Corollary \ref{ContinuityToCaratheodory} and \propref{modcon2} it is sufficient to do the following: for every non-decreasing computable function $t:\NN\to \NN$ with $t(n)\to\infty$
construct a Jordan domain $W$ with Carath{\'e}odory modulus $\eta$ such that
$$\eta(2^{-t(n)})\geq 2^{-n}. $$
To this end, let us set $z_k=\exp(i2^{-k})$, $r_k=\min(\frac14,2^{-t(k)})$, and denote
$$Q_k\equiv Q(z_k,2^{-k-2},r_k). $$
Then
$$ W\equiv \DD\setminus (\cup_{j=1}^\infty Q_j)$$
has the desired property.
\hspace{0.9\textwidth}$\Box$

\bibliographystyle{plain}

\end{document}